\documentclass[12pt]{amsart}
\input amssymb.sty

\newtheorem{thm}{Theorem}
\newtheorem{cor}[thm]{Corollary}
\newtheorem{lem}[thm]{Lemma}

\newtheorem{prop}[thm]{Proposition}

\theoremstyle{definition}
\newtheorem{df}[thm]{Definition}

\newtheorem{rem}[thm]{Remark}

\numberwithin{equation}{section}

\begin{document}

\begin{center}
{\large\bf  J-REGULAR RINGS WITH INJECTIVITIES}

\vspace{0.8cm} { LIANG SHEN
 }
\end{center}

\noindent Abstract. A ring $R$ is called a $J$-regular ring if $R/J(R)$ is von Neumann regular, where $J(R)$ is the Jacobson radical of $R$.
It is proved that if $R$ is $J$-regular, then (i) $R$ is right $n$-injective if and only if every homomorphism from an $n$-generated small right ideal of $R$ to $R_{R}$ can be extended to one from $R_{R}$ to $R_{R}$; (ii) $R$ is right \emph{FP}-injective if and only if $R$ is right ($J, R$)-\emph{FP}-injective. Some known results are improved.

\renewcommand{\thefootnote}{}

\bigskip
\section{\bf Introduction}
\indent Throughout this paper rings are associative with
identity.  A ring $R$ is  $regular$ means it is a von Neumann regular ring. We write $J$  and $S_{r}$ for the Jacobson radical $J(R)$ and the right socle of  $R$, respectively. Let $U$ be a set and $n\geq 1$, $U_{n}$ denotes the set of all $n\times 1$ matrices with entries in $U$. A right ideal $L$ of $R$ is called $small$ if, for any proper right ideal $K$ of $R$, $L+K\neq R$. \\
\indent Recall that a ring $R$ is $right$ \emph{n-injective} if every homomorphism from
 an $n$-generated right ideal of $R$ to $R_{R}$ can be extended to one from $R_{R}$ to
 $R_{R}$. $R$ is  $right$ \emph{F-injective} if $R$ is right $n$-injective for every $n\geq 1$. And $R$ is \emph{right FP-injective} if every  homomorphism from
 a finitely generated submodule of a free right $R$-module $F_{R}$ to $R_{R}$ can be extended to one from $F_{R}$ to
 $R_{R}$. The left side of the above injectivities can be defined similarly. By restricting the ideals to small ones, in [{\bf 6}], the above injectivities are studied under the condition that $R$ is a semiperfect ring with an essential right socle. In [{\bf 5}], the condition is weakened to that $R$ is a semiregular ring. In this short article, the above two conditions are generalized to the one that $R$ is a $J$-regular ring.
Better results are obtained.

\section{\bf Results }
\begin{df}
 A ring $R$ is  $J$-$regular$ if $R/J$
is regular. It is obvious that regular rings are $J$-regular. But the converse is not true. For example, let {$R=\left[
\begin{array}{cc}
\mathbb{Q} & \mathbb{R} \\
0 & \mathbb{Q}
\end{array}
\right]$} be the ring of upper triangular real matrices with all diagonal entries rational. Then
{$J(R)=\left[\begin{array}{cc}
0 & \mathbb{R} \\
0 & 0
\end{array}\right]$}. It is easy to see that $R$ is $J$-regular but not regular.
\end{df}
\begin{rem} {\rm Recall that a ring $R$ is \emph{semilocal} if $R/J$ is a semisimple ring. $R$ is  \emph{semiperfect} in case $R$ is semilocal and idempotents lift modulo $J$. $R$ is  \emph{semiregular} when $R$ is $J$-regular and idempotents lift modulo $J$.  So we have the following relations:
\begin{center}
semiperfect $\Rightarrow$semilocal$\Rightarrow$$J$-regular,\\
semiperfect $\Rightarrow$semiregular$\Rightarrow$$J$-regular.
\end{center}
It is easy to show that $J$-regular rings are real generalizations of the above classes of rings. For example, let $R_{1}$ be a semilocal ring which is not semiregular and $R_{2}$ a semiregular ring that is not semilocal. Set $R=R_{1}\prod R_{2}$. Since $R_{1}$ and $R_{2}$ are both $J$-regular, $R$ is $J$-regular by the following Proposition 5. But $R$ is neither semilocal nor semiregular.
}\end{rem}
\indent A right ideal $I$ of a ring $R$ has a $weak$ $supplement$ in $R$ if there exists a right ideal $K$ of $R$
such that $I+K=R$ and $I\cap K$ is a small right ideal of $R$.
\begin{prop}
The following are equivalent for a ring R:\\
{\rm (1)} $R$ is J-regular.\\
{\rm (2)} Every principal  right (or left) ideal of $R$ has a weak supplement in R.\\
{\rm (3)} Every finitely generated right (or left) ideal of $R$ has a weak supplement in R.

\end{prop}
\begin{proof}
(1)$\Leftrightarrow$(2) is obtained by [{\bf 3}, Proposition 3.18]. It is obvious that (3)$\Rightarrow$(2).
 For (1)$\Rightarrow$(3), suppose that $I$ is a finitely generated right ideal of $R$. Set $\overline{R}=R/J$.
  Since $R$ is $J$-regular, $\overline{I}$ is a direct summand of $\overline{R}$. Then it is easy to get there is a right ideal $K$ of
   $R$ such that $I+K=R$ and $I\cap K\subseteq J$. Therefore, $K$ is a weak supplement of $I$ in $R$.
\end{proof}
\begin{prop}
If R is J-regular, then every factor ring S of R is also J-regular.
\end{prop}
\begin{proof}
Let $S$ be a factor ring of $R$ and $\phi$ be the ring epimorphism from $R$ to $S$. By [{\bf 1}, Corollary 15.8],
$\phi(J)\subseteq J(S)$. So $S/J(S)$ is a factor ring of $R/J$. Since $R/J$ is regular,  $S/J(S)$ is regular. Thus, $S$ is $J$-regular.
\end{proof}
\begin{prop}
 A direct product of
rings R = $\prod_{i\in I} R_{i}$ is J-regular if and only if $
R_{i}$ is J-regular for every $i\in I$.
\end{prop}
\begin{proof}
By [{\bf 2}, Lemma 4.1], it is easy to see that $J$=$\prod_{i\in I} J_{i}$ where $J=J(R)$ and $J_{i}=J(R_{i})$, $i\in I$.  Since
 $\frac{R}{J}=\frac{\prod_{i\in I} R_{i}}{\prod_{i\in I} J_{i}}\cong\prod_{i\in I} \frac{R_{i}}{J_{i}}$,
 $R/J$ is regular if and only if $R_{i}/J_{i}$ is regular for every $i\in I$. So
 $R$ is $J$-regular if and only if $R_{i}$ is $J$-regular for every $i\in I$.
\end{proof}
The following two propositions show that being $J$-regular is a Morita invariant.
\begin{prop}
If $R$ is J-regular, then eRe is also J-regular, where $e^{2}=e\in R$.
\end{prop}
\begin{proof}
We only need to show that for each $a\in eRe$, there exist $b\in eRe$ and $c\in$ $J(eRe)=eJe$ (see [{\bf 2}, Theorem 21.10]) such that $a=aba+c$. As $R$ is $J$-regular, there exist $b'\in R$ and $c'\in J$ such that $a=ab'a+c'$. Since $a\in eRe$, $a=ab'a+c'=aeb'ea+c'$. It is clear that $c'=a-ab'a\in eRe\cap J=eJe$. Then we can set
$b=eb'e$ and $c=c'$.
\end{proof}
\begin{prop}
If R is J-regular, then every matrix ring {\rm M}$_{n\times n}(R)$ is also J-regular, $n\geq 1$.
\end{prop}
\begin{proof}
 It is well-known that $J({\rm M}_{n\times n}(R))={\rm M}_{n\times n}(J)$ (see [{\bf 2}, Page 61]). And it is also easy to prove that $\frac{{\rm M}_{n\times n}(R)}{{\rm M}_{n\times n}(J)}\cong {\rm M}_{n\times n}(\frac{R}{J})$. Therefore  $\frac{{\rm M}_{n\times n}(R)}{J({\rm M}_{n\times n}(R))}=\frac{{\rm M}_{n\times n}(R)}{{\rm M}_{n\times n}(J)}\cong{\rm M}_{n\times n}(\frac{R}{J})$. Since $R$ is $J$-regular, $R/J$ is a regular ring. So  M$_{n\times n}(\frac{R}{J})$ is also regular. Thus, M$_{n\times n}(R)$ is $J$-regular.
\end{proof}
\begin{thm}
Let R be a J-regular ring and  K  a finitely generated projective right $R$-module. Then the endomorphism ring {\rm End} {\rm(}$K${\rm)} of $K$ is also J-regular.
\end{thm}
\begin{proof}
Since $K$ is finitely generated and projective, $K$ is a direct summand of a finitely generated free right $R$-module $F$.
Then there exists some integer $n\geq 1$ such that End($F$)$\cong$M$_{n\times n}(R)$ and End($K$)$\cong$$e$M$_{n\times n}(R)e$ for some idempotent $e$ in M$_{n\times n}(R)$. Thus, by Proposition 6 and Proposition 7, End($K$) is $J$-regular.
\end{proof}
Now we turn to the main results. The following lemma is inspired by [{\bf 3}, Lemma 3.4].
\begin{lem}
Let R be a ring, $b, r_{i}, a_{i}\in R$, $i=1,2,\ldots,n$, such that $b+\sum_{i=1}^{n}a_{i}r_{i}=1$.
Then $bR\cap \sum_{i=1}^{n}a_{i}R=\sum_{i=1}^{n}ba_{i}R$.
\end{lem}
\begin{proof}
Assume that $x\in bR\cap \sum_{i=1}^{n}a_{i}R$. And set $c=\sum_{i=1}^{n}a_{i}r_{i}$. Then there exist $t, t_{1}, \ldots, t_{n}\in R$ such that $x=bt=(1-c)t=\sum_{i=1}^{n}a_{i}t_{i}$. Thus $t=ct+\sum_{i=1}^{n}a_{i}t_{i}\in\sum_{i=1}^{n}a_{i}R$.
So $x=bt\in\sum_{i=1}^{n}ba_{i}R$. Conversely,  $\sum_{i=1}^{n}ba_{i}R=\sum_{i=1}^{n}(1-c)a_{i}R\in bR\cap \sum_{i=1}^{n}a_{i}R$.
\end{proof}
\begin{cor}{\rm ([{\bf 3}, Lemma 3.4])}
Let R be a ring, $r, a\in R$ and $b=1-ar$. Then $bR\cap aR=baR$.
\end{cor}
\begin{thm}If R is J-regular and $n\geq 1$, then R is right n-injective if and only if every homomorphism from an n-generated small right ideal of R to
$R_{R}$ can be extended to one from $R_{R}$ to $R_{R}$.
\end{thm}
\begin{proof}
The necessity is obvious. For the sufficient part, assume that $I=a_{1}R+\cdots +a_{n}R$ is an $n$-generated right ideal of $R$ and $f$ is a homomorphism from $I$ to $R_{R}$.
Since $R$ is $J$-regular, by Proposition 3, $I$ has a weak supplement in $R$. Thus, there exists a right ideal $K$ of $R$ such that $I+K=R$ and $I\cap K\subseteq J$. It is easy to see there are $r_{1}, \ldots, r_{n}\in R$, $b\in K$ such that $b+\sum_{i=1}^{n}a_{i}r_{i}=1$ and
$I\cap bR\subseteq I\cap K\subseteq J$. Therefore, $I\cap bR$ is a small right ideal of $R$. By Lemma 9, $I\cap bR=\sum_{i=1}^{n}ba_{i}R$ is $n$-generated. Thus, by hypothesis, there is a homomorphism $g$ from $R_{R}$ to $R_{R}$ such that $g_{\mid I\cap bR}=f_{\mid I\cap bR}$. Since $I+bR=R$, for each
 $x\in R$, there exist $x_{1}\in I$, $x_{2}\in bR$ such that $x=x_{1}+x_{2}$. Define a map $F$ from $R_{R}$ to $R_{R}$ with
 $F(x)=f(x_{1})+g(x_{2})$ for each $x\in R$. It is easy to prove that $F$ is a well-defined homomorphism from $R_{R}$ to $R_{R}$ such that $F_{\mid I}=f$.
\end{proof}
\begin{cor}
If R is J-regular, then R is right F-injective if and only if every homomorphism from a finitely generated small right ideal of R to
$R_{R}$ can be extended to one from $R_{R}$ to $R_{R}$.
\end{cor}
Let $I$, $K$ be two right ideals of a ring $R$ and $m\geq 1$. In [{\bf 6}], $R$ is called a $right$
($I,K$)-$m$-$injective$ ring if, for any
$m$-generated right ideal $U\subseteq I$, every homomorphism
from $U$ to $K$ can be extended to one from $R_{R}$ to
$R_{R}$. And $R$ is \emph{right {\rm(}I, K{\rm)}-FP-injective} if, for any $n\geq 1$ and any
finitely generated right $R$-submodule $N$ of $I_{n}$ which is a submodule of the free right $R$-module $R_{n}$, every
homomorphism from $N$ to $K$ can be extended to one from $R_{n}$ to $R_{R}$.\\
\indent Using the same method in the proof of
 Theorem 11, we have the following result.
 \begin{thm}
 Let $K$ be a right ideal of a J-regular ring $R$ and $m\geq1$.
 Then R is right ($R,K$)-$m$-$injective$ if and only if R is right ($J,K$)-$m$-$injective$.
 \end{thm}
 \begin{lem}{\rm ([{\bf 6}, Lemma 1.3])}
 The following are equivalent for
two right ideals I and K of a ring R:\\
{\rm (1)} R is right {\rm(}I, K{\rm)}-FP-injective.\\
{\rm (2)} {\rm M}$_{n\times n}${\rm(}R{\rm)} is right
{\rm (M}$_{n\times n}{\rm(}I{\rm)}$, ${\rm M}_{n\times n}{\rm(}K{\rm)}${\rm)}-1-injective
for every n $\geq$ 1.
\end{lem}
\begin{thm}
 If R is J-regular, then R is right
FP-injective if and only if R is right
{\rm(}J, R{\rm)}-FP-injective.
\end{thm}
\begin{proof}
If $R$ is right \emph{FP}-injective, it is clear that $R$ is right {\rm(}$J,R${\rm)}-\emph{FP}-injective. Conversely, assume that $R$ is right {\rm(}$J,R${\rm)}-\emph{FP}-injective. By Lemma 14,  M$_{n\times n}(R)$ is right
({\rm M}$_{n\times n}{\rm(}J{\rm)}$, ${\rm M}_{n\times n}{\rm(}R{\rm)}${\rm)}-1-injective
for every $n \geq$ 1. Since $R$ is $J$-regular,  by Proposition 7, M$_{n\times n}(R)$ is $J$-regular. Again since
$J$(M$_{n\times n}(R)$)=M$_{n\times n}(J)$, Theorem 11 implies that M$_{n\times n}(R)$ is right 1-injective for every $n \geq$ 1.
Thus, by [{\bf 4}, Theorem 5.41], $R$ is right \emph{FP}-injective.
\end{proof}
By the above theorems, we obtain the following corollaries.
\begin{cor}Let R be a semilocal ring.\\
{\rm (1)} If I is a right ideal of R and $m\geq 1$, then R is right (R, I)-m-injective if and only
if R is right (J, I)-m-injective.\\
{\rm (2)} R is right F-injective if and only if $R$ is right (J, R)-n-injective for every $n\geq 1$.\\
{\rm (3)} R is right FP-injective if and only if R is right (J, R)-FP-injective.
\end{cor}
\begin{rem}{\rm Recall that a ring $R$ is  \emph{right small injective} if every homomorphism from a small right ideal of $R$ to $R_{R}$ can be extended to one from $R_{R}$ to $R_{R}$. It was proved in [{\bf 5}, Theorem 3.16 (1)] that if $R$ is semilocal, then $R$ is right self-injective if and only if $R$ is right small injective. But the results in the above corollary weren't obtained in [{\bf 5}].}
\end{rem}
\begin{cor}{\rm ([{\bf 5}, Theorem 3.16 (3), (4)])} Let R be a semiregular ring and $m\geq 1$.\\
{\rm (1)} If I is a right ideal of R, then R is right (J, I)-m-injective if and only
if R is right (R, I)-m-injective.\\
{\rm (2)} R is right (J, R)-FP-injective if and only if R is right FP-injective.

\end{cor}
\begin{cor}{\rm ([{\bf 6}, Lemma 2.3])} Let R be a semiregular ring.\\
{\rm (1)} If R is right (J, $S_{r}$)-1-injective, then R is right (R, $S_{r}$)-1-injective.\\
{\rm (2)} If R is right (J, R)-1-injective, then  R is right 1-injective.
\end{cor}
\begin{cor}{\rm ([{\bf 6},  Theorem 2.11 (1), (2)])} Let R be a semiperfect ring with an essential right socle and $m\geq 1$.\\
{\rm (1)} If R is right (J, $S_{r}$)-m+1-injective, then R is right (R, $S_{r}$)-m-injective.\\
{\rm (2)} If R is right (J, R)-m+1-injective, then R is right m-injective.

\end{cor}
\bigskip
\begin{center}
{\rm  Acknowledgements}
 \end{center}
\indent The article was written during the author's visiting the center of ring theory and its applications in Department of
Mathematics, Ohio University. He would like to thank the center for the hospitality. The research is supported by China Scholarship Council and Southeast University Foundation (No.4007011034 and  No.9207012402). The author is also partially supported  by the National Natural Science Foundation of China (No.10971024).

\begin{flushleft}

LIANG SHEN\\
DEPARTMENT OF MATHEMATICS\\
SOUTHEAST UNIVERSITY\\
NANJING, 210096\\
P.R.CHINA\\
E-mail: lshen@seu.edu.cn
\end{flushleft}

\end{document}